\documentclass[11pt]{amsart}
\usepackage{amsmath,amssymb,amsthm}
\usepackage[margin=1.1in]{geometry}
\usepackage{booktabs}
\usepackage{graphicx}
\usepackage[colorlinks=true,linkcolor=blue,citecolor=blue,urlcolor=blue]{hyperref}
\newtheorem{theorem}{Theorem}[section]
\newtheorem{lemma}[theorem]{Lemma}
\newtheorem{proposition}[theorem]{Proposition}

\theoremstyle{definition}
\newtheorem{definition}[theorem]{Definition}
\theoremstyle{remark}
\newtheorem{remark}[theorem]{Remark}
\newtheorem{example}[theorem]{Example}
\newcommand{\Q}{\mathbb{Q}}
\newcommand{\Z}{\mathbb{Z}}
\newcommand{\F}{\mathbb{F}}
\newcommand{\A}{\mathfrak{A}}
\newcommand{\Qbar}{\overline{\Q}}
\newcommand{\NS}{\mathrm{NS}}
\newcommand{\fr}[1]{\langle #1\rangle}
\DeclareMathOperator{\Gal}{Gal}
\DeclareMathOperator{\tr}{tr}
\DeclareMathOperator{\Pic}{Pic}

\title[Galois-invariant N\'eron--Severi ranks of Fermat surfaces]{Galois-invariant N\'eron--Severi ranks
of Fermat surfaces over number fields: a Galois module, closed forms, a threshold, and exact tables}
\author{Rifat Jumagulov}
\email{jum.rifm@gmail.com}
\date{July 2026}

\begin{document}
\begin{abstract}
For a Fermat surface $X_d\colon x_0^d+x_1^d+x_2^d+x_3^d=0$ we give a character-theoretic formula for
the Galois-invariant
N\'eron--Severi (Picard) rank $\rho_K(X_d)=\dim\bigl(\NS(X_{\Qbar})\otimes\Q\bigr)^{\Gal(\Qbar/K)}$ for
every subfield $K\subseteq\Q(\zeta_d)$, and reduce the case of an arbitrary number field $K$ to
$\rho_K=\rho_{K\cap L}$ (intersection inside a fixed $\Qbar$), where $L$ is the field of definition of
the geometric Picard group determined by Gvirtz-Chen--Skorobogatov \cite{GCS}. The $\zeta_d$-rational N\'eron--Severi group is described
explicitly as a monomial Galois module, making $\rho_K$ a character average --- evaluated in closed
form for $\gcd(d,6)=1$, and per degree for $4\le d\le30$. For degree coprime to $6$ we prove
$\rho_\Q(X_d)=1+3(\Psi_2(d)-3\tau(d)+2)$ with $\Psi_2$ multiplicative, so $\rho_\Q(X_p)=3p-5$ for prime
$p\ge5$; we establish a field-of-definition threshold with explicit Hasse--Davenport witnesses and an
assembled orbit rule for even $d$; and we tabulate $\rho_\Q(X_d)$ exactly for $4\le d\le30$, the
exceptional entries at $d=14,24,28,30$ being proved by Hasse--Davenport identities and a stabilizer
descent, and corroborated by exact $\Z[\zeta_m]$ evaluation and by the per-character computation of
\cite{GCS}. In summatory form, $\sum_{d\le x,\gcd(d,6)=1}\rho_\Q(X_d)\sim\tfrac35 x^2$. All results
are unconditional: on a surface the algebraic classes are the rational $(1,1)$-classes by Lefschetz. Ancillary
files provide standalone exact-arithmetic scripts and machine-readable certificates.
\end{abstract}
\maketitle

\section{Relation to prior work}\label{sec:intro}
The \emph{geometric} Picard number $\rho_{\Qbar}$ of a Fermat surface is classical: Shioda
\cite{Shioda79,Shioda81}, Aoki \cite{Aoki83}, Aoki--Shioda \cite{AS83}. The lines generate
$\NS(X_d)\otimes\Q$ for $\gcd(d,6)=1$ (Shioda \cite[Thm.~7]{Shioda81}, restated in
\cite[Thm.~3.1]{SSvL}; proved for every $d$ coprime to $6$ by Degtyarev \cite{Degtyarev} and computationally for
$d\le100$ by Sch\"utt--Shioda--van Luijk \cite{SSvL}); the full equivalence ``$d\le4$ or $\gcd(d,6)=1$'' folds in the
classical small degrees ($d\le3$ classical, $d=4$ due to Mizukami \cite{Mizukami}) and is stated as such
in \cite[Thm.~3.1]{SSvL}. The \emph{field of definition} $L$ of the geometric Picard group is determined
for arbitrary $d$ by Gvirtz-Chen--Skorobogatov \cite{GCS}: $L$ is the compositum of $\Q(\mu_{2d})$,
$\Q(2^{2/d})$ if $2\mid d$, $\Q(3^{3/d})$ if $3\mid d$, and finitely many fields $M_{d'}$, indexed by
the exceptional divisor levels $d'$ of $d$ ($12\le d'\le180$),
with $L=\Q(\mu_{2d})$ exactly when $d\le4$ or $\gcd(d,6)=1$.

What is new here is the \emph{rank} --- and the evaluated form of the module. Two layers should be
distinguished. Given the classical eigenspace decomposition \cite{Shioda79} and the per-character
description of \cite{GCS} (the Gr\"ossencharacter), the monomial-module formalism of \S\ref{sec:setup} and
the invariant averaging over intermediate fields (Theorem~\ref{thm:J1}) are formal
representation-theoretic consequences, and no novelty is claimed for the formalism itself; likewise
the reduction $\rho_K=\rho_{K\cap L}$ is a formal corollary of the existence of the field $L$,
recorded here for use. What this paper adds is the \emph{evaluation}: GCS determine which field $L$
trivialises the Galois
action but do not compute the Galois-invariant rank $\rho_K$. To the best of the author's knowledge,
no previous work gives a closed formula of this form, or a table, for the Galois-invariant rank of the
Fermat family over $\Q$; the adjacent strands --- geometric Picard numbers, line generation, fields of
definition, individual diagonal quartics --- are cited above and below, and among the works cited here
the $\rho_K$ evaluations are confined to degree $4$. For
$d=4$ an explicit character decomposition of $\NS\otimes\Q$ appears only for the Dwork quartic at the
$\rho=19$ stratum (Duan \cite{Duan}), and the Brauer-group literature (Ieronymou--Skorobogatov--Zarhin
\cite{ISZ}) treats $\NS$ as a Galois module only for diagonal quartics; also at $d=4$, the Picard group
of diagonal quartic surfaces over $\Q$ and its Galois cohomology have a prior computational literature
via the $48$ lines (Pinch--Swinnerton-Dyer \cite{PSD}, Swinnerton-Dyer \cite{SD2}, Bright
\cite{Bright}). We give: a monomial description
of the Galois module for all $d$, explicitly evaluated in the stated regimes, with $\rho_K$ a character
average for $K\subseteq\Q(\zeta_d)$ ---
explicit wherever the per-degree character data is computed, i.e.\ in closed form for $\gcd(d,6)=1$ and
per degree for $4\le d\le30$ --- and reduced to the
field-of-definition case for every number field $K$ (\S\ref{sec:module}), the closed form of
\S\ref{sec:odd}, the threshold theorem of \S\ref{sec:threshold}, the unconditional table $4\le d\le30$
(\S\ref{sec:tables}), and the first average order known to the author (\S\ref{sec:asymp}). A companion paper treats the
codimension-$n/2$ analogue on higher-dimensional Fermat varieties, where the invariant is no longer a
Picard rank.

\section{Setup and the reduction}\label{sec:setup}
Let $X=X_d\subset\mathbb{P}^{3}$ be the Fermat surface, with the $\mu_d^{4}/\mu_d$-action and eigenspace
decomposition $H^2_{\mathrm{prim}}(X,\mathbb{C})=\bigoplus_a V(a)$ over characters $a=(a_1,\dots,a_4)$,
$a_i\in(\Z/d)\smallsetminus\{0\}$, $\sum a_i\equiv 0\ (d)$ \cite{Shioda79}. The \emph{algebraic characters}
are
\[
\A \;=\;\Bigl\{a \;:\; \textstyle\sum_i\fr{t a_i/d}=2\ \ \forall t\in(\Z/d)^\times\Bigr\},
\]
where $\fr\cdot$ is the fractional part; by the Lefschetz $(1,1)$ theorem $\bigoplus_{a\in\A}V(a)$ is the complexification of the primitive
part of $\NS(X)\otimes\Q$, so all results below are unconditional.

At a split prime $p\equiv1\ (d)$, with $\chi$ a character of order $d$ on $\F_p^\times$ and
$g(\eta)=\sum_{x\in\F_p^\times}\eta(x)\,e^{2\pi ix/p}$ the Gauss sum, the Frobenius acts on $V(a)$ by a
normalized Jacobi product whose \emph{twist}
$u_a(p)=\prod_i g(\chi^{a_i})/p^{2}$ lies in $\mu_w$, $w=d$ for even $d$ and $w=2d$ for odd $d$ (the
torsion of $\Q(\zeta_d)^\times$). (Conventions: Frobenius means the \emph{geometric} Frobenius on
$H^2_{\mathrm{\acute et}}$; replacing it by the arithmetic one inverts every $u_a$, which permutes
each orbit's twist values and affects no statement below --- cf.\ the choice-independence in
Definition~\ref{rem:terms}. By Weil's theorem the Jacobi sums define a Gr\"ossencharacter of $\Q(\zeta_d)$ \cite{Weil52};
that the normalized value $u_a(p)$ moreover has \emph{finite order} for algebraic $a$ follows from
the unit argument (the two steps are separated exactly this way in \cite[\S6]{GCS}): every
archimedean absolute value of $u_a(p)$ is $1$ (as $|g(\eta)|=\sqrt p$); it is an algebraic unit ---
above $p$ by Stickelberger together with the Hodge condition, away from $p$ since
$g(\eta)\,g(\eta^{-1})=\eta(-1)\,p$ --- hence by Kronecker a root of unity of $\Q(\zeta_d)$,
i.e.\ in $\mu_w$.) Characters $a$ are \emph{ordered} tuples (characters of
$\mu_d^4/\mu_d$); coordinate permutations are not quotiented, and all orbits below are
$(\Z/d)^\times$-scaling orbits of tuples. Frobenius reciprocity for the resulting
monomial Galois representation (Proposition~\ref{prop:orbitcontrib}) gives the counting rule used
throughout:
\begin{equation}\label{eq:orbitrule}
\rho_\Q \;=\; 1+\#\bigl\{(\Z/d)^\times\text{-orbits of }\A\text{ with trivial stabilizer
cocycle-character and twist}\bigr\},
\end{equation}
where a single split prime witnessing $u_a(p)\neq1$ eliminates the orbit of $a$, and the stabilizer
character of an orbit fixed by $t$ --- the quadratic cocycle $\varepsilon$ of Theorem~\ref{thm:assembled},
derived in Lemma~\ref{lem:sublevel} and Proposition~\ref{prop:cocycle} --- is evaluated on lifts of $t$
to $(\Z/2d)^\times$ for even $d$ (a modulus sufficient for the parity cocycle; see \S\ref{sec:even}).

\begin{definition}[Terminology and the split-prime criterion]\label{rem:terms}
Five terms recur. The \emph{vertical twist} (Gr\"ossencharacter twist) of $a$ is the finite-order character $\chi_a$ of
$\Gal(\Qbar/\Q(\zeta_d))$ defined above, recorded through its Frobenius values
$\chi_a(\mathrm{Frob}_{\mathfrak p})=u_a(p)$ at the primes $\mathfrak p$ of $\Q(\zeta_d)$ above split
$p\equiv1\ (d)$ (every such $\mathfrak p$ is unramified and of good reduction, since $p\nmid d$, and
these degree-one primes have density $1$ among the primes of $\Q(\zeta_d)$; replacing the auxiliary character $\chi$ by
$\chi^s$, $s\in(\Z/d)^\times$ --- equivalently, changing the choice of $\mathfrak p$ above $p$ ---
replaces $a$ by $sa$, so all orbit-level statements are independent of the choice); by Chebotarev
density $\chi_a$ is trivial iff $u_a(p)=1$ at every split prime, so a single split prime with
$u_a(p)\neq1$
\emph{proves} nontriviality, while triviality is only ever certified by an exact identity. The
\emph{orbit ceiling} is $1+\#(\A/(\Z/d)^\times)$, the value of $\rho_\Q$ if every orbit contributed its
maximal $1$; it is an upper bound, each orbit contributing $0$ or $1$
(Proposition~\ref{prop:orbitcontrib}).
A decomposable character is \emph{$\Sigma$-even} if its pair-representative sum $\Sigma x(a)$
(canonical representatives $x<d/2$, self-paired $d/2$ counted once per pair) is even; this parity is
independent of all choices (Theorem~\ref{thm:parity}). An \emph{exceptional} character is one in the
exceptional stratum of the classification (\S\ref{sec:prov}): neither decomposable nor on a family.
The \emph{exceptional layer} at level $d$
consists of the vertically-trivial non-decomposable orbits; its contribution to $\rho_\Q$ is
denoted $\mathcal E_{d}$ (the coefficient field $\Q(\zeta_d)$ is written $E$ --- the two do not
collide).
\end{definition}

\begin{remark}[The invariant dimension is the arithmetic rank]\label{rem:rank}
Throughout, $\rho_K$ is the dimension of the Galois invariants, and it equals the rank of $\NS(X_K)$:
by Hochschild--Serre the cokernel of $\NS(X_K)\to\NS(X_{\Qbar})^{\Gal(\Qbar/K)}$ embeds in the
(torsion) Brauer group $\mathrm{Br}(K)$, so the two agree after $\otimes\,\Q$; and $\Pic=\NS$ here
over any field of definition, since $H^1(X_d,\mathcal O)=0$ for a surface in $\mathbb P^3$, so
$\Pic^0$ is trivial.
\end{remark}

\begin{proposition}[Orbit contribution]\label{prop:orbitcontrib}
Let a finite group $G$ act transitively on a finite set of one-dimensional eigenlines, with the
stabilizer $G_a$ of a line $V(a)$ acting on it by a character $\varepsilon_a$. Then the orbit module
is isomorphic to the induced module $\operatorname{Ind}_{G_a}^{G}\varepsilon_a$, and
\[
\dim\bigl(\operatorname{Ind}_{G_a}^{G}\varepsilon_a\bigr)^{G}\;=\;\dim\varepsilon_a^{G_a}\in\{0,1\},
\]
equal to $1$ iff $\varepsilon_a$ is trivial on $G_a$.
\end{proposition}
\begin{proof}
Frobenius reciprocity:
$\operatorname{Hom}_G(\mathbf 1,\operatorname{Ind}_{G_a}^{G}\varepsilon_a)
=\operatorname{Hom}_{G_a}(\mathbf 1,\varepsilon_a)$.
\end{proof}
\noindent Every orbit count below is an instance of this rule applied to the vertically-trivial layer;
a vertically nontrivial orbit contributes $0$ before the rule is reached, having no invariants already
under $\Gal(\Qbar/\Q(\zeta_d))$.

\smallskip
\noindent For reference, the recurring notation:
\begin{center}\small
\begin{tabular}{@{}ll@{}}
\toprule
$\A$ & algebraic characters (Hodge condition); orbit ceiling $=1+\#(\A/(\Z/d)^\times)$\\
$V(a)$ & the $a$-eigenline ($\mathbb C$-eigenspace in \S\ref{sec:setup}; its $E$-form from \S\ref{sec:even} on)\\
$u_a$, $\chi_a$ & the vertical (Gr\"ossencharacter) twist and its global character (Def.~\ref{rem:terms})\\
$\mathcal S$ & the $\Q(\zeta_d)$-rational algebraic characters (trivial vertical twist, \S\ref{sec:even})\\
$\varepsilon_a$ & the stabilizer cocycle character (\S\ref{sec:even}; derived in Prop.~\ref{prop:cocycle})\\
$E$;\ $\mathcal E_d$ & the coefficient field $\Q(\zeta_d)$; the exceptional-layer contribution to $\rho_\Q$\\
$L$;\ $L_\chi$, $M_{d'}$ & field of definition of $\NS(X_{\Qbar})$; per-character twist fields, their compositum \cite{GCS}\\
$k$;\ $\Sigma x(\cdot)$ & $k=(p-1)/d$; the pair-representative sum (Def.~\ref{rem:terms})\\
$m$ & a level: exceptional $m=2q$ (Thm~\ref{thm:qr}); sub-level $\gcd(t-1,d)$ (Lemma~\ref{lem:sublevel}, local)\\
$\psi$, $\rho$ & the odd-part (order-$q$) and quadratic characters at level $m=2q$ (Thm~\ref{thm:qr})\\
\bottomrule
\end{tabular}
\end{center}

\section{Odd degree: dichotomy and the closed form}\label{sec:odd}

\begin{lemma}[Pairing lemma]\label{lem:pairing}
If $\chi$ has odd order then $\chi(-1)=1$ and $g(\chi^x)g(\chi^{-x})=p$ for $x\not\equiv0$.
Consequently every \emph{decomposable} character (one admitting a perfect matching into zero-sum pairs)
has twist exactly $1$ at every split prime.
\end{lemma}
\begin{proof}
$\chi(-1)^2=1$ and $\chi(-1)\in\mu_d$ with $d$ odd force $\chi(-1)=1$; the identity
$g(\chi)g(\chi^{-1})=\chi(-1)p$ is standard. A decomposable $4$-tuple is two such pairs, so
$\prod_i g(\chi^{a_i})=p^{2}$.
\end{proof}

\begin{theorem}[Dichotomy]\label{thm:dichotomy}
For $d\ge4$, $\rho_\Q(X_d)$ equals the orbit ceiling $1+\#(\A/(\Z/d)^\times)$ iff $\gcd(d,6)=1$.
\end{theorem}
\begin{proof}
$(\Leftarrow)$ For $\gcd(d,6)=1$ the classification of \S\ref{sec:prov} gives $\A=$ the decomposable set,
and Lemma~\ref{lem:pairing} applies, so every orbit has trivial twist; every stabilizer scalar is
trivial as well (the sub-level character at odd level is trivial by Lemma~\ref{lem:pairing}, via
Lemma~\ref{lem:sublevel}), so every orbit contributes; hence
$\rho_\Q=1+\#(\A/(\Z/d)^\times)$.
$(\Rightarrow)$ Suppose $\gcd(d,6)>1$. At $d=4$ the claim is a direct computation: $\rho_\Q(X_4)=5$
while the ceiling is $1+10=11$ (\S\ref{sec:tables}, corroborated geometrically by the line computation
of \S\ref{sec:prov}). For $d\ge5$ we exhibit one algebraic orbit with nontrivial twist, which by the
counting rule \eqref{eq:orbitrule} contributes $0$ and so forces $\rho_\Q<$ the ceiling. If $2\mid d$ (so
$d\ge6$), family~(a) at $x=1$, namely $a=(2,\,d-1,\,\tfrac d2-1,\,\tfrac d2)$, is nondecomposable with twist
the Kummer character of $2^{1/(d/2)}$ (Theorem~\ref{thm:families}); it is nontrivial since $d/2\ge3$ makes
$\Q(2^{1/(d/2)})/\Q$ non-normal, hence not contained in the abelian field $\Q(\mu_{2d})$. If $d$ is odd with
$3\mid d$ (so $d\ge9$), the progression $a=(1,\,1+\tfrac d3,\,1+\tfrac{2d}3,\,-3)$ has twist $\chi(3)^{-3}$,
the Kummer character of $3^{1/(d/3)}$, nontrivial by the same argument (Theorem~\ref{thm:threshold}(2)).
Either way an orbit is lost, giving $\rho_\Q<$ the ceiling. (Equivalently, and independently: if every orbit
contributed its maximal $1$ then $\Gal(\Qbar/\Q(\mu_{2d}))$ would fix each one-dimensional $V(a)$, so
$\NS(X_{\Qbar})$ would be defined over $\Q(\mu_{2d})$, i.e.\ $L\subseteq\Q(\mu_{2d})$; but \cite{GCS} give
$L=\Q(\mu_{2d})$ only for $d\le4$ or $\gcd(d,6)=1$, so for $\gcd(d,6)>1$ with $d\ge5$ (the case
$d=4$ being the direct computation above) some orbit has a nontrivial
character and $\rho_\Q<$ the ceiling.)
\end{proof}

\smallskip
\noindent The case split at $d=4$ is forced by the witness, not the conclusion: family (a) at $x=1$
degenerates there to $(2,3,1,2)$, which is decomposable --- its twist is the cyclotomic parity
character $(-1)^{k}$ of Theorem~\ref{thm:parity} ($\Sigma=3$), not a non-abelian Kummer character
(at $d=4$ the radical $2^{1/(d/2)}=\sqrt2$ generates a normal extension) --- so the
Kummer-nontriviality route needs $d\ge5$.

\begin{theorem}[Closed form]\label{thm:closedform}
For $\gcd(d,6)=1$,
\[
\rho_\Q(X_d)\;=\;1+3\bigl(\Psi_2(d)-3\tau(d)+2\bigr),\qquad
\Psi_2(p^k)=\sigma(p^k)+\sigma(p^{k-1})
\]
($\Psi_2$ multiplicative with $\Psi_2(1)=1$, $\sigma(n)=\sum_{e\mid n}e$, $\tau$ the divisor count). In
particular $\rho_\Q(X_p)=3p-5$ for prime
$p\ge5$, and $\rho_\Q(X_{25})=91$, $\rho_\Q(X_{35})=160$, $\rho_\Q(X_{49})=175$.
\end{theorem}
\begin{proof}
By Theorem~\ref{thm:dichotomy} and the classification, $\rho_\Q=1+N_{\mathrm{orb}}$ with
$N_{\mathrm{orb}}$ the orbit count of the decomposable set. Inclusion--exclusion over the three
pair-partitions of four indices gives $|\A_d|=3(d-1)(d-2)$ (the triple intersection needs
$2x\equiv0$, impossible for odd $d$). A character fixed by $t$ has entries in the $m$-torsion,
$m=\gcd(t-1,d)$, and decomposability transfers to level $m$; hence $\mathrm{Fix}(t)=3(m-1)(m-2)$ and
Burnside gives $N_{\mathrm{orb}}=\frac{3}{\varphi(d)}\sum_t\bigl(g_t^2-3g_t+2\bigr)$, $g_t=\gcd(t-1,d)$.
The sums $\sum_t g_t=\tau(d)\varphi(d)$ and $\sum_t g_t^2=\varphi(d)\Psi_2(d)$ are multiplicative by CRT.
\end{proof}

\section{The field-of-definition threshold}\label{sec:threshold}

\begin{theorem}\label{thm:threshold}
Let $d$ be odd.
\begin{enumerate}
\item If $\gcd(d,6)=1$ (equivalently, every prime divisor of $d$ exceeds $3$), then every algebraic
character is decomposable and $\rho_\Q=$ the orbit ceiling.
\item If $3\mid d$ and $d\neq3$, then, for every $x$ such that all four entries are nonzero, the character
$a=\{x,\,x+d/3,\,x+2d/3,\,-3x\}$ is
algebraic, and its twist equals $\chi(3)^{-3x}$ exactly; for $x=1$ this is the Kummer character
of $3^{1/m}$, $m=d/3\ge3$, which is nontrivial. Hence $\rho_\Q<$ the ceiling: $3\mid d$ breaks it.
\end{enumerate}
\end{theorem}
\begin{proof}
(1) is the classification input (\S\ref{sec:prov}) plus Lemma~\ref{lem:pairing}. For (2): algebraicity
follows from the Hermite identity $\sum_{j=0}^{2}\fr{y+j/3}=\fr{3y}+1$ applied
at $y=tx/d$: the fourth entry contributes $\fr{-3tx/d}$, and since $-3x\not\equiv0\ (d)$, $3tx/d$ is
non-integral for every unit $t$, so $\fr{3tx/d}+\fr{-3tx/d}=1$ and the four fractional parts sum
to $2$. For the twist, the Hasse--Davenport product relation with $\psi=\chi^{d/3}$ of order
$3$ gives
\[
\prod_{j=0}^{2}g(\chi^x\psi^j)=\chi(3)^{-3x}\,g(\chi^{3x})\,g(\psi)g(\psi^2),
\]
and $g(\psi)g(\psi^2)=p$ by
Lemma~\ref{lem:pairing}; with the partner $g(\chi^{-3x})$ the total is
$\chi(3)^{-3x}p^{2}$. Nontriviality: $u$ trivial at all split $p$ forces
$3^{1/m}\in\Q(\zeta_d)$ by Kummer theory and Chebotarev, impossible for $m\ge3$ since
$\Q(3^{1/m})/\Q$ is non-normal and subfields of cyclotomic fields are abelian. (For even $d$ the
argument fails at $m=2$: the radical is $\sqrt3$, which lies in the cyclotomic field
$\Q(\zeta_{12})$ (as $\sqrt{-3}\in\Q(\zeta_3)$ and $i\in\Q(\zeta_4)$) --- so the parity of $d$ is
load-bearing.)
\end{proof}

\begin{remark}[Higher dimension]\label{rem:higherdim}
The same mechanism --- the progression $\{x,\,x+d/\ell,\dots,x+(\ell-1)d/\ell,\,-\ell x\}$, padded with
zero-sum pairs, for an odd prime $\ell\le n+1$ dividing $d$ --- operates on $n$-dimensional Fermat
varieties in codimension $n/2$, where the invariant is no longer a Picard rank and the classification
input differs; the general-$n$ threshold (``every prime divisor of $d$ exceeds $n+1$'') is treated in
the companion paper.
\end{remark}

\begin{proposition}[Assembled rule, odd degree divisible by $3$]\label{prop:odd3}
Let $d$ be odd with $3\mid d$. Every decomposable orbit contributes $1$: its twist is trivial by
Lemma~\ref{lem:pairing}, and every stabilizer scalar is trivial, since the sub-level character of a
decomposable $b$ at odd level $m$ is trivial (the general sub-level mechanism of
Lemma~\ref{lem:sublevel}; at odd $m$ the parity law is vacuous --- $\Q(\zeta_{2m})=\Q(\zeta_m)$ --- and
Lemma~\ref{lem:pairing} at level $m$ gives the triviality). Every valid family-(b) member has nontrivial twist
(Lemma~\ref{lem:famnontriv}, odd case) and contributes $0$. The exceptional layer is supported on the
divisor levels in $\{15,21\}$. Hence
\[
\rho_\Q(X_d)\;=\;1+N_{\mathrm{orb}}(d)+\mathcal E_{d},\qquad
N_{\mathrm{orb}}(d)=\frac{3}{\varphi(d)}\sum_{t\in(\Z/d)^\times}\bigl(g_t^2-3g_t+2\bigr),\quad
g_t=\gcd(t-1,d),
\]
as in the proof of Theorem~\ref{thm:closedform}. For $d\le30$ ($d=9,15,21,27$) the exact evaluation
gives $\mathcal E_{d}=0$: every indecomposable orbit at levels $15$ and $21$ already has nontrivial vertical
twist, witnessed \emph{exactly} in $\Z[\zeta_m]$ at the single split primes $p=31$ and $p=43$
respectively ($36/36$ orbits each; ancillary \texttt{exceptional\_exact\_twist.py}), so
$\rho_\Q=1+N_{\mathrm{orb}}$:
$31,\,76,\,106,\,130$, matching \S\ref{sec:tables}.
\end{proposition}

\section{Even degree}\label{sec:even}

\begin{theorem}[Parity law]\label{thm:parity}
For even $d$ and $k=(p-1)/d$, a decomposable character with pair representatives $x_i$ has twist exactly
$(-1)^{k\sum x_i}$ --- the quadratic character of $\Q(\zeta_{2d})/\Q(\zeta_d)$.
\end{theorem}
\begin{proof}
$\chi(-1)=\zeta_d^{(p-1)/2}=(\zeta_d^{d/2})^{k}=(-1)^k$; each pair contributes $\chi^{x_i}(-1)$.
Representative-independence: $x\mapsto d-x$ shifts $\sum x_i$ by an even amount. Matching-independence
(when a character admits more than one perfect matching): the value $u_a(p)$ is intrinsic, and split
primes with $k$ odd exist ($p\equiv1+d\ (2d)$, Dirichlet), so the parity of $\sum x_i$ is pinned by
$u_a$ and agrees for any two matchings.
\end{proof}

\begin{theorem}[Family twists]\label{thm:families}
The three Aoki--Shioda one-parameter families (\S\ref{sec:prov}) have exact twists
\[
\text{(a)}\ \chi(2)^{2x}(-1)^{\frac{p-1}{2}},\qquad
\text{(b)}\ \chi(3)^{-3x}\,\chi^{3x}(-1),\qquad
\text{(c)}\ \chi(2)^{6x-d/2}(-1)^{\frac{p-1}{2}}.
\]
In particular the radicals $2^{2/d},3^{3/d}$ of \cite{GCS} are precisely the family Kummer characters.
\end{theorem}
\begin{proof}
Write $\rho=\chi^{d/2}$ (quadratic) and use $g(\eta)g(\eta^{-1})=\eta(-1)p$ for $\eta\neq1$,
$g(\rho)^2=\rho(-1)p=(-1)^{(p-1)/2}p$, and the Hasse--Davenport duplication
$g(\psi)g(\psi\rho)=\overline\psi(4)\,g(\psi^2)g(\rho)$. Family (b) is the $\ell=3$ progression of
Theorem~\ref{thm:threshold}(2): the product relation gives
\[
\prod_{j=0}^{2}g(\chi^{x+jd/3})=\chi(3)^{-3x}g(\chi^{3x})\,g(\psi_3)g(\psi_3^2),\qquad
g(\psi_3)g(\psi_3^2)=p\quad(\psi_3=\chi^{d/3}),
\]
and the partner factor is
$g(\chi^{3x})g(\chi^{-3x})=\chi^{3x}(-1)\,p$, whence $u_b=\chi(3)^{-3x}\chi^{3x}(-1)$: the reflection
factor is trivial for odd $d$ (where $\chi(-1)=1$, recovering Theorem~\ref{thm:threshold}(2)'s
formula) and equals $(-1)^{xk}$, $k=(p-1)/d$, for even $d$. For family (a) $\{2x,-x,-x+\tfrac d2,\tfrac
d2\}$, duplication on $\{-x,-x+\tfrac d2\}$ gives $g(\chi^{-x})g(\chi^{-x}\rho)=\chi(2)^{2x}g(\chi^{-2x})
g(\rho)$, whence
\[
u_a=\chi(2)^{2x}\,\bigl[g(\chi^{2x})g(\chi^{-2x})\bigr]\,g(\rho)^2/p^2
=\chi(2)^{2x}\,\chi^{2x}(-1)\,\rho(-1)=\chi(2)^{2x}(-1)^{(p-1)/2},
\]
using $\chi^{2x}(-1)=\chi(-1)^{2x}=1$. Family (c) $\{4x,-2x,-x+\tfrac d4,-x+\tfrac{3d}4\}$ ($4\mid d$): put
$\sigma=\chi^{d/4}$ (order $4$, $\sigma^2=\rho$), so $\chi^{-x+3d/4}=\chi^{-x}\sigma\rho$; two duplications
give $g(\chi^{-x}\sigma)g(\chi^{-x}\sigma\rho)=\chi(4)^x\sigma(4)^{-1}g(\chi^{-2x}\rho)g(\rho)$ and
$g(\chi^{-2x})g(\chi^{-2x}\rho)=\chi(4)^{2x}g(\chi^{-4x})g(\rho)$, whence
\begin{align*}
u_c&=\chi(4)^{3x}\sigma(4)^{-1}\bigl[g(\chi^{4x})g(\chi^{-4x})\bigr]g(\rho)^2/p^2\\
&=\chi(2)^{6x}\,\chi(2)^{-d/2}\,(-1)^{(p-1)/2}=\chi(2)^{6x-d/2}(-1)^{(p-1)/2},
\end{align*}
using $\sigma(4)=\chi(2)^{d/2}$ and $\chi^{4x}(-1)=1$. The non-cyclotomic component of each twist is
the Kummer character of $2^{2/d}$
(resp.\ $3^{3/d}$), matching the radicals of \cite{GCS}; for even $d$ the family-(b) twist carries in
addition the quadratic \emph{cyclotomic} reflection factor $\chi^{3x}(-1)$, and families (a), (c) the
displayed $(-1)^{(p-1)/2}$. The closed forms are corroborated numerically
(absolute discrepancy below $10^{-30}$) at several primes per degree \emph{of both parities of
$k=(p-1)/d$} (ancillary \texttt{family\_twists.py}; an even-$k$-only battery is blind to the
reflection factor).
\end{proof}

\begin{lemma}[Family twists are nontrivial off the degenerate strata]\label{lem:famnontriv}
Let $a$ be a valid member (all entries nonzero) of family (a), (b), or (c) with parameter $x$. Via
$(-1)^{(p-1)/2}=\chi(-1)^{d/2}$ and $\chi^{3x}(-1)=\chi_p((-1)^{x})$, each family twist of
Theorem~\ref{thm:families} is the Kummer character
$p\mapsto\chi_p(\beta)$ of a single rational number:
$\beta_{\mathrm a}=4^{x}(-1)^{d/2}$, $\beta_{\mathrm b}=(-27)^{-x}$, and $\beta_{\mathrm c}=2^{\,6x-d/2}$
(where $4\mid d$, so $(-1)^{d/2}=1$). The twist is nontrivial --- hence a single split prime eliminates
the orbit --- for every $x$ except the following degenerate cases:
\begin{enumerate}
\item family (a), $4\mid d$, $x\in\{d/4,\,3d/4\}$: the member is \emph{decomposable}
($\{d/2,\,3d/4,\,d/4,\,d/2\}$ and its conjugate) and belongs to the decomposable stratum;
\item family (b), $d$ even, $3x\equiv d/2\ (d)$ (solutions $x\equiv d/6 \pmod{d/3}$): the member is
\emph{decomposable}. For odd $d$ there is no degenerate case: \emph{every} valid family-(b) member at
odd degree has nontrivial twist;
\item family (c), $6x\equiv d/2\ (d)$: the member is \emph{decomposable} ($3x\equiv d/4\ (d)$ pairs
$-2x$ with $-x+\tfrac d4$, and $3x\equiv 3d/4\ (d)$ pairs $-2x$ with $-x+\tfrac{3d}4$; the remaining
pair is zero-sum); and family (c), $6x\equiv0\ (d)$: the twist is the quadratic character of $\sqrt2$,
nontrivial iff $8\nmid d$. For $8\mid d$ (which with the requisite $3\mid d$ forces $24\mid d$; the
valid parameters are $x\in\{\tfrac d6,\tfrac d3,\tfrac{2d}3,\tfrac{5d}6\}$) these members are
non-decomposable with \emph{trivial} vertical twist ($\sqrt2\in\Q(\zeta_d)$); they are counted with the
vertically-trivial exceptional orbits and decided by the stabilizer descent, i.e.\ within $\mathcal E_{d}$. At
$d\le30$ this occurs exactly at $d=24$: the orbits of $(2,14,16,16)$ and $(8,8,10,22)$
($x=4,8,16,20$), all removed there by the stabilizer descent $\varepsilon(13)=-1$
(Lemma~\ref{lem:sublevel}(ii), worked in the $d=24$ passage of Theorem~\ref{thm:qr}).
\end{enumerate}
\end{lemma}
\begin{proof}
Folding the parity/reflection factor into the base, $u_a(p)=\chi_p(\beta)$ with $\beta$ as displayed
($\chi_p(-1)^{d/2}=(-1)^{(p-1)/2}$ as in Theorem~\ref{thm:parity}; for (b),
$\chi^{3x}(-1)=\chi_p(-1)^{x}$, giving the sign of $(-27)^{-x}$). By Kummer theory and Chebotarev,
$u\equiv1$ at all split primes iff some $d$-th root of $\beta$ lies in $E=\Q(\zeta_d)$. Any $d$-th root
is $2^{e/d}\zeta$ (resp.\ $3^{e/d}\zeta$) with $\zeta^d=\pm1$, i.e.\ $\zeta\in\mu_{2d}$, where $e$ is
the exponent of $2$ (resp.\ $3$) in $\beta$; so triviality puts $2^{e/d}$ (resp.\ $3^{e/d}$) in the
abelian field
$\Q(\zeta_{2d})$, and $\Q(2^{e/d})=\Q(2^{1/(d/g)})$ (resp.\ $\Q(3^{1/(d/g)})$), $g=\gcd(e,d)$, is
normal only for $d/g\le2$. Thus
a nondegenerate $x$ (with $d/g\ge3$) has nontrivial twist. The boundary $d/g\le2$ is enumerated per
family. For (a),(b): $e\equiv0\ (d)$ is entry-invalid, and $e\equiv d/2\ (d)$ gives the listed
$x$-values, whose members are checked decomposable directly. For (c): $e\equiv0\ (d)$ is
$6x\equiv d/2$, the listed decomposable members; $e\equiv d/2\ (d)$ is $6x\equiv0$, where $\beta_{\mathrm
c}$ is $2^{d/2}$ times a $d$-th power, with $d$-th root $\sqrt2\,\zeta$, $\zeta\in\mu_d\subset E$
($\beta_{\mathrm c}>0$), so the twist is trivial iff $\sqrt2\in\Q(\zeta_d)$ iff $8\mid d$; for
$8\nmid d$ it is the nontrivial quadratic character of $\sqrt2$ over $E$. Decomposability of the
listed degenerate members and
non-decomposability of the $24\mid d$ stratum are finite checks ($(2,14,16,16)$: no pair of entries
sums to $0$ mod $24$).
\end{proof}

\begin{theorem}[Assembled rule]\label{thm:assembled}
For even $d$,
\[
\rho_\Q(X_d)\;=\;1+\#\mathcal{O}_d+\mathcal E_{d},
\]
where:
\begin{enumerate}
\item $\mathcal{O}_d$ is the set of $\Sigma$-even decomposable orbits all of whose stabilizer sub-level
characters have even $\Sigma x(b)$, the stabilizer character on a lift $t'$ being
$\varepsilon(t')=(-1)^{((t'-1)/m)\,\Sigma x(b)}$;
\item $\mathcal E_{d}$ is the contribution of the vertically-trivial non-decomposable orbits (the
exceptional orbits proper together with the degenerate family-(c) stratum of
Lemma~\ref{lem:famnontriv}, which requires $24\mid d$), nonzero only when a divisor of $d$ lies in the
$22$-element exceptional list, and decided per degree by a finite exact procedure (the per-orbit
$\Z[\zeta_m]$ twist evaluation and the stabilizer descent).
\end{enumerate}
It reproduces $\rho_\Q(X_4)=5$, $\rho_\Q(X_6)=14$, $\rho_\Q(X_8)=14$, $\rho_\Q(X_{10})=26$.
\end{theorem}
\begin{proof}
For even $d$ the algebraic set $\A$ is covered by the decomposable characters, the three one-parameter
families (a)--(c), and finitely many exceptional Galois orbits (\S\ref{sec:prov}); we make this a
\emph{disjoint} stratification by assigning any decomposable character to the decomposable stratum, and
any non-decomposable family member with trivial vertical twist --- the $24\mid d$ stratum of
Lemma~\ref{lem:famnontriv}(3) --- to the exceptional stratum. (A
family or exceptional member can itself be decomposable for small $d$ --- e.g.\ at $d=12$ the family-(c)
character $\{4,10,2,8\}=\{4,8\}\sqcup\{10,2\}$ --- in which case it is governed by the parity law
(Theorem~\ref{thm:parity}; here the pair representatives are $4$ and $2$, $\Sigma=6$ even, so the twist
is $1$), not by Theorem~\ref{thm:families}.)
By Proposition~\ref{prop:orbitcontrib} a monomial orbit $\operatorname{Ind}_H^G\theta$ contributes
$\dim(\operatorname{Ind}_H^G\theta)^G\in\{0,1\}$, equal to $1$ iff the orbit's Gr\"ossencharacter twist
\emph{and} its stabilizer character are both trivial. For a decomposable orbit this holds iff its
representative has $\sum x_i$ even \emph{and} the quadratic stabilizer cocycle
$\varepsilon(t')=(-1)^{((t'-1)/m)\sum x(b)}$ (Theorem~\ref{thm:parity}, on a lift $t'$ to $(\Z/2d)^\times$;
derived in Lemma~\ref{lem:sublevel} and Proposition~\ref{prop:cocycle} below) is trivial on every fixing $t'$; these form $\mathcal O_d$. A \emph{genuinely-family}
(non-decomposable) member of the family stratum has nontrivial Kummer twist by
Lemma~\ref{lem:famnontriv} and contributes $0$. A remaining exceptional
orbit contributes iff its twist and stabilizer character (order dividing $2d$) are both trivial; this can
occur only when a divisor of $d$ lies in the $22$-element list \cite{Terasoma}, and for each such degree the
twist is an exact root of unity in $\Z[\zeta_d]$ decided case by case --- for $d\le30$, Theorem~\ref{thm:qr}
and its $m=24,28,30$ arguments decide every surviving class internally, and the remaining orbits are
eliminated by direct exact evaluation (rigorous for \emph{non}triviality; ancillary
\texttt{exceptional\_exact\_twist.py}); the resulting finite, per-degree
count is $\mathcal E_{d}$. Summing the
three disjoint contributions and the hyperplane gives $\rho_\Q=1+\#\mathcal O_d+\mathcal E_{d}$, e.g.\ $\rho_\Q(X_4)=5$,
$\rho_\Q(X_6)=14$, $\rho_\Q(X_8)=14$, $\rho_\Q(X_{10})=26$; the exceptional term $\mathcal E_{d}$ is a finite per-degree
count, not a closed form (see \S\ref{sec:tables} for $d\le30$).
\end{proof}

\begin{theorem}[The doubling/kernel telescope]\label{thm:qr}
An exceptional character at level $m=2q$ ($q$ odd), with all entries odd, has twist identically $1$
when its $\psi$-exponent multiset is invariant under
doubling.
Precisely, write $m=2q$ with $q$ odd; an entry equal to $q=m/2$ has $\psi$-exponent $0$ (trivial order-$q$
part) and belongs to the $2$-part. If all entries are odd and the \emph{multiset} $S$ of
\emph{nonzero} $\psi$-exponents in $(\Z/q)\smallsetminus\{0\}$ (the reductions mod $q$ of the entries not divisible
by $q$, taken with multiplicity; they need not be units mod $q$) is invariant under $r\mapsto2r$ --- equivalently, is a union of
complete orbits of the doubling map $r\mapsto2r$ (a bijection of $(\Z/q)\smallsetminus\{0\}$, as $2$
is a unit mod $q$) with constant multiplicity on each; in every application below the
exponents are distinct, so $S$ is a plain union of such orbits --- the twist telescopes to $1$ (the doubling
telescope).
In particular the $m{=}14$ characters $(1,7,9,11)$, $(3,5,7,13)$ --- whose nonzero $\psi$-exponent sets
$\{1,2,4\},\{3,5,6\}$ (the entry $7=q$ excluded, its $\psi$-exponent being $0$) are the quadratic
residues/non-residues in $(\Z/7)^\times$, each a $\langle2\rangle$-coset --- telescope to $1$, whence
$\mathcal E_{14}=8$ and $\rho_\Q(X_{14})=46$ unconditionally; the same telescope at $m=30$, a sub-level reduction
at $m=28$, and a direct Hasse--Davenport identity at $m=24$ (see the proof) give
$\mathcal E_{30}=6$, $\mathcal E_{28}=8$, $\mathcal E_{24}=6$ unconditionally as well.
\end{theorem}
\begin{proof}
The twist $u_a=\prod_i g(\chi^{a_i})/p^{2}$ is a root of unity. Write $m=2q$ ($q$ odd), $\rho=\chi^{q}$ (order
$2$), and let $\psi$ (order $q$) satisfy $\chi=\rho\psi$; then for an odd exponent $a$ one has
$\chi^{a}=\rho\,\psi^{s}$ with $s=a\bmod q$, so $g(\chi^{a})=g(\rho\psi^{s})$, and an entry $a=q=m/2$ gives
$s=0$, i.e.\ the factor $g(\rho)$. Assume all four entries odd, and let $S$ be the multiset of \emph{nonzero}
$\psi$-exponents. The Hasse--Davenport duplication relation
$g(\psi^{s})\,g(\rho\psi^{s})=\overline{\psi^{s}}(4)\,g(\psi^{2s})\,g(\rho)$ gives
$g(\rho\psi^{s})=\overline{\psi^{s}}(4)\,g(\rho)\,g(\psi^{2s})/g(\psi^{s})$; multiplying over $S$, when $S$ is
invariant under $s\mapsto2s$ as a multiset (in particular, a plain union of complete doubling-orbits) the ratios
$\prod_{s\in S}g(\psi^{2s})/g(\psi^{s})$ telescope to $1$, and $\sum_{s\in S}s\equiv\sum_i a_i\equiv0\ (q)$
(the zero-sum condition $\sum_i a_i\equiv0$) makes
$\prod_{s\in S}\overline{\psi^{s}}(4)=\overline{\psi}^{\,\sum_S s}(4)=1$. Hence
$\prod_{s\in S}g(\rho\psi^{s})=g(\rho)^{|S|}$. An exceptional character has at most one entry equal to
$q$ (two such entries leave a zero-sum pair, i.e.\ a decomposable character), so $|S|\in\{3,4\}$: with
the $g(\rho)$ from the $s=0$ entry when present,
$\prod_i g(\chi^{a_i})=g(\rho)^{|S|+r}=g(\rho)^{4}=\bigl(\rho(-1)p\bigr)^{2}=p^{2}$ ($r\in\{0,1\}$
entries equal to $q$, $|S|+r=4$), so $u_a=1$. At $m=14$ ($q=7$) the
exceptional characters $(1,7,9,11),(3,5,7,13)$ have nonzero $\psi$-exponent sets $\{1,2,4\},\{3,5,6\}$ (the
entry $7=q$ excluded) --- the quadratic residues and non-residues mod $7$, each a $\langle2\rangle$-orbit
closed under doubling --- so both give $u_a=1$; and the tuple stabilizers of these characters are trivial
(the $48$ ordered characters of the two multisets fall into $8$ orbits of full length $\varphi(14)=6$;
$(3,5,7,13)\equiv3\cdot(1,7,9,11)$ up to reordering), so no stabilizer condition arises, giving $\mathcal E_{14}=8$.

The three remaining exceptional degrees are now decided \emph{internally}. (In each case the displayed
list of vertically-trivial orbits is exhaustive: every other indecomposable orbit at the level is
eliminated by an exact single-prime nontriviality witness in $\Z[\zeta_m]$ --- rigorous, as one failing
prime is a proof; ancillary \texttt{exceptional\_exact\_twist.py}.)
\emph{At $m=30$} ($q=15$; the $2$-part is exactly $2$) the telescope applies verbatim: the class
$(1,17,19,23)$ has all entries odd, none equal to $15$, and nonzero $\psi$-exponent set
$\{1,2,4,8\}=\langle2\rangle\subset(\Z/15)^\times$, so $u_a=1$ identically (the $r=0$ case above); its
tuple stabilizer is trivial (the $48$ ordered characters of the multisets $\{1,17,19,23\}$,
$\{7,11,13,29\}$ fall into $6$ full-length orbits), so all six contribute: $\mathcal E_{30}=6$.
\emph{At $m=28$} every vertically-trivial orbit is imprimitive, $a=2b$ with $b$ in the level-$14$
exceptional class just treated: by the $\pi_2^\ast$ identification (Lemma~\ref{lem:sublevel}) the twist
of $a$ equals that of $b$, trivial by the $m=14$ telescope; the only nontrivial tuple-stabilizer element
is $t=15\equiv1\ (14)$, and $\varepsilon_a(15)$ is the value at $\sigma_{15}$ of the \emph{trivial}
level-$14$ character of $b$, hence $+1$ --- so all eight orbits contribute: $\mathcal E_{28}=8$.
\emph{At $m=24$} the $12$ vertically-trivial orbits split into two $(\Z/24)^\times$-scaling families of
six. The six family-(c)-type orbits of $(2,14,16,16)$ and $(8,8,10,22)$ (the degenerate stratum of
Lemma~\ref{lem:famnontriv}(3)) are each stabilized by $t=13$, whose
sub-level $b=a/2$ at level $12$ is family-(c) with Kummer radical $\sqrt2\notin\Q(\zeta_{12})$, so
$\varepsilon(13)=\sigma_{13}(\sqrt2)/\sqrt2=-1$ (Lemma~\ref{lem:sublevel}(ii)) removes them. For the six
orbits of $(1,11,17,19)$, apply the Hasse--Davenport product relation (displayed in the proof of
Theorem~\ref{thm:threshold}; the relation holds for any multiplicative character) with
$\psi=\chi^{8}$ of order $3$, at $x=9$ and at $x=3$:
\[
g(\chi^{9})g(\chi^{17})g(\chi^{1})=\chi(3)^{-27}g(\chi^{3})\,g(\psi)g(\psi^{2}),\qquad
g(\chi^{3})g(\chi^{11})g(\chi^{19})=\chi(3)^{-9}g(\chi^{9})\,g(\psi)g(\psi^{2});
\]
multiplying and cancelling $g(\chi^{3})g(\chi^{9})$,
\[
\prod_i g(\chi^{a_i})=g(\chi^{1})g(\chi^{11})g(\chi^{17})g(\chi^{19})
=\chi(3)^{-36}\bigl(g(\psi)g(\psi^{2})\bigr)^{2}=\rho(3)\,p^{2}=p^{2},
\]
since $\chi^{-36}=\chi^{12}=\rho$ with $\rho(3)=\bigl(\tfrac3p\bigr)=1$ for $p\equiv1\ (12)$, and
$g(\psi)g(\psi^{2})=p$ by Lemma~\ref{lem:pairing}. Hence $u_a=1$ identically; the tuple stabilizers are
trivial, so all six contribute: $\mathcal E_{24}=6$.

Thus $\mathcal E_{14}=8$, $\mathcal E_{24}=6$, $\mathcal E_{28}=8$, $\mathcal E_{30}=6$ are theorems of this paper. They are corroborated
independently by the exact $\Z[\zeta_m]$ evaluation of every Jacobi-sum twist (ancillary
\texttt{exceptional\_exact\_twist.py}: a single split prime rigorously proves \emph{non}triviality, and
the trivial twists reproduce identically), and by the \emph{proven} per-character field-of-definition
computation of \cite{GCS} (their kernels $L_\chi$, \cite[Prop.~6.4, Rem.~6.6]{GCS}, whose published
aggregate compositum $M_{d'}=\prod_\chi L_\chi$ is consistent with the counts (the compositum alone
does not determine the number of trivial individual fields); we reproduce their character
classification prime-free in ancillary \texttt{gcs\_fieldofdef\_grounding.py}). A complete scan of
all exceptional levels up to $180$ is outside the scope of the present article.
\end{proof}

\subsection*{Derivation of the stabilizer cocycle (Theorem~\ref{thm:assembled})}

Write $G=\mathrm{Gal}(\Q(\zeta_d)/\Q)=(\Z/d)^\times$ and $\sigma_t$ for the automorphism
$\zeta_d\mapsto\zeta_d^{\,t}$.

\smallskip
\noindent\textbf{The descent as an averaged trace.} By Shioda's eigenspace decomposition each
algebraic eigenspace has a natural $E$-form: for algebraic $a$ we write $V(a)$, from here on, for
the $a$-eigenline of $\NS(X_{\Qbar})\otimes_\Z E$ --- a line over $E=\Q(\zeta_d)$ whose
complexification is the eigenspace of \S\ref{sec:setup}; for $a\in\mathcal S$ it lies in
$\NS(X_E)\otimes_\Z E$ --- and, after extension of scalars (as in
\cite[proof of Prop.~6.3]{GCS}),
\[
\NS\bigl(X_{E}\bigr)\otimes_\Z E \;=\; E h\ \oplus\ (M\otimes_\Q E),\qquad
M\otimes_\Q E=\bigoplus_{a\in\mathcal S} V(a),
\]
where $h$ is the hyperplane class and $M$ denotes the primitive part of $\NS(X_E)\otimes\Q$ --- a $\Q$-form of the right-hand sum ---
and $\mathcal S\subset\A$ is the set of $\Q(\zeta_d)$-rational algebraic characters: those
whose vertical (Gr\"ossencharacter) twist is trivial at every split prime, i.e.\ by the parity law
(Theorem~\ref{thm:parity}) the $\Sigma$-even decomposable characters together with the trivial-twist
exceptionals. The $E$-linear extension $\sigma_t\otimes1$ of the $G$-action (acting on the
N\'eron--Severi factor only) permutes the lines,
$\sigma_t\colon V(a)\to V(ta)$, and when $ta=a$ it acts on $V(a)$
by a root of unity $\varepsilon_a(t)$ (the action on $\NS\otimes\Q$ factors through a finite quotient
of $\Gal(\Qbar/\Q)$, so each stabilizer scalar has finite order); traces are unchanged by extension of
scalars, so
$\chi_M(t):=\operatorname{tr}(\sigma_t\mid M)=\operatorname{tr}(\sigma_t\mid M\otimes E)$. Thus
$M\otimes E=\bigoplus_{[a]}\operatorname{Ind}_{G_a}^{G}\varepsilon_a$
with $G_a=\mathrm{Stab}_G(a)$, and Proposition~\ref{prop:orbitcontrib} gives
\begin{equation}\label{eq:descenttrace}
\rho_\Q \;=\; 1+\dim M^{G} \;=\; 1+\frac1{|G|}\sum_{t\in G}\operatorname{tr}(\sigma_t\mid M)
\;=\; 1+\#\bigl\{\,[a]\ :\ \varepsilon_a|_{G_a}\equiv1\,\bigr\}.
\end{equation}
(The middle expression is the per-residue signed trace; computed via the Jacobi-sum Frobenius
eigenvalues it is Chebotarev-constant on each class of $G$, and reproduces
$\rho_{\Q(\zeta_d)}=1+\chi_M(\sigma_1)$
and every intermediate $\rho_K$ of \S\ref{sec:module}.) It remains to identify $\varepsilon_a$ on $G_a$.

\begin{lemma}[Sub-level reduction]\label{lem:sublevel}
Let $a\in\mathcal S$ (vertical triviality is needed only so that $\varepsilon_a$ is well defined on $G$ ---
two absolute lifts of $t$ differ by an element of $\Gal(\Qbar/\Q(\zeta_d))$, which acts trivially
on $V(a)$; the $\pi_f^\ast$ isomorphism below holds for any $a\in\A$) and let $t\in G_a$ fix $a$
entrywise, $ta_i\equiv a_i\ (d)$. Put
$m=\gcd(t-1,d)$ (a sub-level of $d$; this local $m$ is unrelated to the exceptional level $m=2q$ of
Theorem~\ref{thm:qr}) and $f=d/m$, so $a_i\in f\Z/d\Z$ and $a=f\,b$ with $b\in(\Z/m)^4$. Then $b$ is an
algebraic character at level $m$, the pullback $\pi_f^\ast$ along the $\Q$-morphism below induces a
$\Gal(\Qbar/\Q)$-equivariant isomorphism
$V_m(b)\otimes_{\Q(\zeta_m)}\Q(\zeta_d)\xrightarrow{\ \sim\ }V(a)$, and $\varepsilon_a(t)$
equals the value at $\sigma_{t'}$ (any lift $t'$ of $t$; note $t'\equiv1\ (m)$) of the finite-order
character of the level-$m$ Gr\"ossencharacter of $b$. In particular:
\textup{(i)} if $a$ (hence $b$) is decomposable, that character is the parity law
(Theorem~\ref{thm:parity} at level $m$), and for any lift $t'$ of $t$ to $(\Z/2d)^\times$
\[
\varepsilon_a(t)\;=\;(-1)^{\,\frac{t'-1}{m}\,\Sigma x(b)},
\]
where $\Sigma x(b)$ is the pair-representative sum of $b$;
\textup{(ii)} if $b$ carries the Kummer character of a radical $\delta^{1/w}$ (a family or exceptional
character, Theorem~\ref{thm:families}, Lemma~\ref{lem:famnontriv}), then
$\varepsilon_a(t)=\sigma_{t'}(\delta^{1/w})/\delta^{1/w}$, computed in the field generated
over $\Q(\zeta_m)$ by the radical; here $\varepsilon_a(t)$ is normalized as the scalar by which
$\sigma_{t'}$ acts on $V(a)=\pi_f^\ast V_m(b)$ (the direction convention is immaterial for the
order-$2$ values used in this paper). (E.g.\ $d=24$, $t=13$, $a=(2,14,16,16)$: $b=(1,7,8,8)$ is
family-(c) at level $12$ with $x=2$, radical $\sqrt2\notin\Q(\zeta_{12})$;
$\sigma_{13}(\sqrt2)=-\sqrt2$ in $\Q(\zeta_{24})$, so $\varepsilon_a(13)=-1$.)
\end{lemma}

\begin{proof}
That $b$ is algebraic at level $m$ is the level-$m$ transfer already used in the proof
of Theorem~\ref{thm:closedform} (the Hodge condition
$\sum_i\langle sa_i/d\rangle=2$ for all $s$ restricts to $\sum_i\langle s'b_i/m\rangle=2$ for all
$s'\in(\Z/m)^\times$; when $a$ is decomposable its pairing pushes to a pairing of $b$). For the Galois
action, use the finite morphism defined over $\Q$
\[
\pi_f\colon X_d\longrightarrow X_m,\qquad [x_0:\dots:x_3]\longmapsto[x_0^{\,f}:\dots:x_3^{\,f}],
\]
which is well defined because $\sum_i(x_i^{\,f})^{m}=\sum_i x_i^{\,fm}=\sum_i x_i^{\,d}$. It intertwines
the $\mu_d^4$-action on $X_d$ with the $\mu_m^4$-action on $X_m$ through the $f$-power map $\mu_d\to\mu_m$,
so pullback carries the character-$b$ eigenline to the character-$fb$ one:
$\pi_f^\ast V_m(b)=V_d(fb)=V(a)$. Since $\pi_{f\ast}\pi_f^\ast=\deg(\pi_f)$, the map $\pi_f^\ast$ is
injective, hence (after extension of scalars to $\Q(\zeta_d)$) an isomorphism
$V_m(b)\otimes_{\Q(\zeta_m)}\Q(\zeta_d)\xrightarrow{\ \sim\ }V(a)$ of one-dimensional
$\Q(\zeta_d)$-spaces; being defined
over $\Q$ it is $\Gal(\Qbar/\Q)$-equivariant, so it identifies the stabilizer scalar $\varepsilon_a(t)$ with the
action of $\sigma_t$ on $V_m(b)$ \emph{with no extra character twist}. On $V_m(b)$ that action is the
value of the finite-order character of the level-$m$ Gr\"ossencharacter of $b$, which cuts out the
abelian extension of $\Q(\zeta_m)$ generated by the corresponding Kummer radical; by Chebotarev its
split-prime values determine it. If $b$ is decomposable, this character is the parity law
(Theorem~\ref{thm:parity} applied to $X_m$), the quadratic character of $\Q(\zeta_{2m})/\Q(\zeta_m)$;
the restriction of a lift $t'$ to that quotient is
$\frac{t'-1}{m}\bmod2$, giving $\varepsilon_a(t)=(-1)^{\frac{t'-1}{m}\Sigma x(b)}$ --- case
\textup{(i)}. If $b$ carries the Kummer character of $\delta^{1/w}$, the value at $\sigma_{t'}$ (which
fixes $\Q(\zeta_m)$, as $t'\equiv1\ (m)$) is $\sigma_{t'}(\delta^{1/w})/\delta^{1/w}$ --- case
\textup{(ii)}. This value is determined by $t'$ alone: since $a\in\mathcal S$ is vertically trivial,
$\Gal(\Qbar/\Q(\zeta_d))$ acts trivially on $V(a)\cong V_m(b)\otimes_{\Q(\zeta_m)}\Q(\zeta_d)$, so the
level-$m$ twist field $\Q(\zeta_m)(\delta^{1/w})$ embeds in $\Q(\zeta_d)$, and the restriction of
$\sigma_{t'}$ to it depends only on $t'\bmod d$ (e.g.\ at $d=24$:
$\Q(\zeta_{12},\sqrt2)=\Q(\zeta_{24})$).
\end{proof}

\begin{lemma}[Lift-independence]\label{lem:liftindep}
With the notation of Lemma~\ref{lem:sublevel} and $a$ $\Sigma$-even, the sign
$(-1)^{\frac{t'-1}{m}\Sigma x(b)}$ is independent of the chosen lift $t'\in(\Z/2d)^\times$ of $t$.
\end{lemma}

\begin{proof}
The two lifts of $t$ are $t'$ and $t'+d$, so the exponent changes by $\frac{d}{m}\,\Sigma x(b)=f\,\Sigma x(b)$;
it suffices that this is even. Scaling by $f=d/m$ is a bijection from the zero-sum pairs of $b$ in $\Z/m$
to those of $a=fb$ in $\Z/d$ that preserves the canonical representative choice: for $v\in\Z/m$,
$v<m/2\iff fv<d/2$, the self-paired value $v=m/2$ maps to the self-paired $fv=d/2$, and
$f\cdot(m-v)=d-fv\equiv-fv\ (d)$. Hence, term by term, $\Sigma x(a)=f\cdot\Sigma x(b)$ (as integers, with
canonical representatives). Because $a$ is $\Sigma$-even, $\Sigma x(a)$ is even, so $f\,\Sigma x(b)=\Sigma x(a)$
is even.
\end{proof}

\begin{proposition}[The cocycle and the assembled rule]\label{prop:cocycle}
For even $d$, a $\Sigma$-even decomposable orbit $[a]$ contributes $1$ to $\rho_\Q$ iff for every
$t\in G_a$ the sub-level character $b=a/f$ (with $f=d/\gcd(t-1,d)$) has $\Sigma x(b)$ even; otherwise it
contributes $0$. Together with the family strata (contribution $0$, Lemma~\ref{lem:famnontriv},
with its $24\mid d$ stratum routed to $\mathcal E_{d}$) and the
exceptional term $\mathcal E_{d}$, \eqref{eq:descenttrace} becomes $\rho_\Q=1+\#\mathcal O_d+\mathcal E_{d}$, which is
Theorem~\ref{thm:assembled}.
\end{proposition}

\begin{proof}
By \eqref{eq:descenttrace} the orbit contributes iff $\varepsilon_a\equiv1$ on $G_a$; by
Lemmas~\ref{lem:sublevel}--\ref{lem:liftindep} this holds iff $\frac{t'-1}{m}\Sigma x(b)$ is even for every
$t\in G_a$, and the condition depends only on $t$, not on the lift. Write $q=\frac{t'-1}{m}$, so
$\gcd(q,f)=1$ for the canonical lift. If $q$ is odd the condition is exactly ``$\Sigma x(b)$ even''. If $q$
is even then $f=d/m$ is odd, and Lemma~\ref{lem:liftindep} together with $\Sigma x(a)$ even give
$f\,\Sigma x(b)=\Sigma x(a)$ even, hence $\Sigma x(b)$ even --- so the condition holds automatically. In both
cases $[a]$ contributes iff every fixing $t$ has $\Sigma x(b)$ even; this is the definition of $\mathcal O_d$.
The family and exceptional strata are unchanged.
\end{proof}

\noindent The identity $\Sigma x(a)=f\,\Sigma x(b)$ was verified in all $2276$ stabilizer pairs $(t,a)$
(over decomposable \emph{characters}, not orbit representatives) for even $d\le30$ (zero violations); the configuration ``$f$ odd and $\Sigma x(b)$ odd'' never occurs on $M$
(confirming lift-independence there) but does occur for $\Sigma$-odd characters (outside $M$), so the
$\Sigma$-even hypothesis is necessary. Proposition~\ref{prop:cocycle} reproduces the decomposable part of the
table exactly at every degree with $\mathcal E_{d}=0$, the residual at $d=14,24,28,30$ being the exceptional
$\mathcal E_{d}=8,6,8,6$ (ancillary \texttt{evendescent\_cocycle\_check.py}). Lemma~\ref{lem:sublevel} rests on
Shioda's eigenspace decomposition and the parity law (case (i)), together with Kummer theory for the
family/exceptional sub-levels (case (ii)).

\begin{example}[The even-degree descent, worked for $d=8$]\label{ex:d8descent}
We exhibit every step of $\rho_{\mathbb{Q}(\zeta_8)}\rightsquigarrow\rho_{\mathbb{Q}}$
for the octic surface $X_8$, making Theorems~\ref{thm:assembled}--\ref{thm:qr}
fully explicit. Here $(\mathbb{Z}/8)^\times=\{1,3,5,7\}\cong C_2\times C_2$.

\smallskip
\noindent\emph{Strata.} Of the $|\A_8|=175$ algebraic characters, $127$ are
decomposable and $48$ lie on the two even families (a),(c) of
Theorem~\ref{thm:families} (no family (b), since $3\nmid8$). At the orbit level this
is $49$ orbits total (orbit ceiling $50$): $12$ family orbits, which carry the
nontrivial Kummer twist $\chi(2)^{\bullet}$ and contribute $0$; and $37$
decomposable orbits. By the parity law (Theorem~\ref{thm:parity}) a decomposable
orbit has twist $(-1)^{k\sum x_i}$, trivial at \emph{every} split prime iff
$\sum x_i$ is even. Exactly $19$ of the $37$ decomposable orbits are
$\Sigma$-even; the other $18$ are eliminated by any split prime with $k=(p-1)/8$
odd (e.g.\ $p=41$).

\smallskip
\noindent\emph{The split-prime-invisible descent.} A split prime $p\equiv1\ (8)$
is a \emph{trivial} element of $(\mathbb{Z}/8)^\times$ (Frobenius acts through
$\mathrm{Gal}(\overline{\mathbb{Q}}/\mathbb{Q}(\zeta_8))$), so it can only witness
the vertical twist $u_a$ and never the action of a nontrivial $t$. The remaining
descent $\mathbb{Q}(\zeta_8)\to\mathbb{Q}$ is governed entirely by the stabilizer
cocycle on $t\neq1$, which no split prime reaches. Concretely: an orbit with
representative $a$ fixed by $t$ (so $ta\equiv a$, $m:=\gcd(t-1,8)$) contributes
iff the sub-level character $b=a/(8/m)\in(\mathbb{Z}/m)^4$ has
$\varepsilon(t')=(-1)^{((t'-1)/m)\sum x(b)}=+1$ on the lift $t'$ to
$(\mathbb{Z}/16)^\times$ ($2d=16$); equivalently, iff $\sum x(b)$ is even.

Among the $19$ $\Sigma$-even decomposable orbits, exactly the six that make up the
$(\mathbb{Z}/8)^\times$-orbit family of the multiset $\{2,4,4,6\}$ are removed:
each is fixed by $t=5$ (with $m=\gcd(4,8)=4$), and its sub-level character
$b=\{1,2,2,3\}\subset\mathbb{Z}/4$ has $\sum x(b)=1+2\equiv1\pmod2$ (the self-paired
$2=d/2$ contributes once, the pair $\{1,3\}$ contributes its representative $1$),
so $\varepsilon(5')=-1$. The other $13$ $\Sigma$-even orbits survive (e.g.\
$\{1,1,7,7\},\{1,3,5,7\},\{2,2,6,6\},\{4,4,4,4\}$; note $\{2,2,6,6\}$ is also fixed
by $t=5$ but has sub-level $\{1,1,3,3\}$ with even $\sum x(b)=2$, so it survives).
Hence
\[
\rho_{\mathbb{Q}}(X_8)\;=\;1+\underbrace{19}_{\Sigma\text{-even dec.\ orbits}}
-\underbrace{6}_{\varepsilon(5')=-1\ \text{removes}}\;=\;14 .
\]

\smallskip
\noindent\emph{Module cross-check.} Equivalently, in the notation of
\S\ref{sec:module}, the signed trace vector of $\sigma_t$ on
$\mathrm{NS}(X_{\mathbb{Q}(\zeta_8)})\otimes\mathbb{Q}$ is
$\bigl(\chi_{\NS}(1),\chi_{\NS}(3),\chi_{\NS}(5),\chi_{\NS}(7)\bigr)=(56,\,2,\,-4,\,2)$,
the single negative entry $\chi_{\NS}(5)=-4$ being precisely the $\varepsilon(5')=-1$
cocycle above. Theorem~\ref{thm:J1} then reproduces, in one stroke, the whole
cyclotomic tower
\[
\rho_{\mathbb{Q}(\zeta_8)}=56,\quad
\rho_{\mathbb{Q}(\sqrt2)}=\rho_{\mathbb{Q}(\sqrt{-2})}=\tfrac{56+2}2=29,\quad
\rho_{\mathbb{Q}(i)}=\tfrac{56-4}2=26,\quad
\rho_{\mathbb{Q}}=\tfrac{56+2-4+2}4=14,
\]
so the descent is what separates $\mathbb{Q}(i)$ (which sees the cocycle,
$\rho=26$) from the other two quadratic subfields ($\rho=29$). The descent mechanism
admits an independent, geometrically disjoint check on the quartic $X_4$, whose $48$
lines generate $\NS$: the line-intersection form and the Galois permutation of lines
alone reproduce the entire tower $\rho_{\mathbb{Q}}(X_4)=5<
\rho_{\mathbb{Q}(i)}=8<\rho_{\mathrm{geom}}=20$ (ancillary
\texttt{lines\_descent\_quartic\_check.py}). For $d=8$ the lines do \emph{not} generate
$\NS(X_8)$ (\S\ref{sec:intro}), so no analogous line check is available;
$\rho_{\mathbb{Q}(i)}(X_8)=26$ rests on the module computation above.
\end{example}

\section{The tables}\label{sec:tables}
The \emph{decomposable} contributions ($4\le d\le30$) are unconditional: the even-degree descent underlying
them is derived in Proposition~\ref{prop:cocycle}, and each is decided by the \emph{exact} twist --- a
Gauss-sum product reduced by Hasse--Davenport to a root of unity in $\Z[\zeta_d]$ --- never by sampling split
primes: a single split prime proves \emph{non}triviality (and so eliminates an orbit), but a
\emph{contribution} is certified only by the exact evaluation (so the one-prime-in-a-trivial-coset artifact
cannot inflate these counts). The \emph{exceptional} contributions $\mathcal E_{d}$ are unconditional and, for
$4\le d\le30$, proved \emph{internally} (Theorem~\ref{thm:qr}): the doubling telescope at $d=14$ and
$d=30$, the sub-level reduction at $d=28$, and the direct Hasse--Davenport identity at $d=24$, with the
stabilizer descent of Proposition~\ref{prop:cocycle} and Lemma~\ref{lem:sublevel}, give
$\mathcal E_{14}=8,\ \mathcal E_{24}=6,\ \mathcal E_{28}=8,\ \mathcal E_{30}=6$ and $\rho_\Q(X_{24})=98,\ \rho_\Q(X_{28})=94,\
\rho_\Q(X_{30})=176$.
Here a vertically-trivial orbit has Frobenius eigenvalue $p$ at every split prime (an algebraic class over
$\Q(\zeta_m)$), so its twist is $1$ identically. Two independent corroborations: the \emph{exact}
$\Z[\zeta_m]$ evaluation of every Jacobi-sum twist (ancillary \texttt{exceptional\_exact\_twist.py};
rigorous for \emph{non}triviality), and the field-of-definition computation of \cite{GCS} --- for $d=14$
the aggregate $M_{14}=\Q(\zeta_{14})$ is trivial (all exceptional classes $\Q(\zeta_{14})$-rational); for $d=24,28,30$
which classes are $\Q(\zeta_d)$-rational is fixed by their per-character kernels $L_\chi$
\cite[Prop.~6.4, Rem.~6.6]{GCS}, which the exact-arithmetic Jacobi evaluation reproduces and whose
published aggregate is consistent with the counts.
Verification scripts (standalone; exact identities, high-precision numerical corroborations with
absolute discrepancy below $10^{-30}$, and a double-precision point-count layer decided by exact
integer equality) accompany
the manuscript.
\begin{center}\resizebox{\textwidth}{!}{$
\begin{array}{c|ccccccccccccc}
d & 4&5&6&7&8&9&10&11&12&13&14&15&16\\\hline
\rho_\Q(X_d) & 5&10&14&16&14&31&26&28&38&34&46&76&35
\end{array}$}\end{center}
\begin{center}\resizebox{\textwidth}{!}{$
\begin{array}{c|cccccccccccccc}
d & 17&18&19&20&21&22&23&24&25&26&27&28&29&30\\\hline
\rho_\Q(X_d) & 46&74&52&62&106&62&64&98&91&74&130&94&82&176
\end{array}$}\end{center}
For even $d$ the decomposable and exceptional parts of the count separate as follows (family
members that are not decomposable contribute $0$ throughout, Lemma~\ref{lem:famnontriv};
for odd $d$, $\rho_\Q=1+N_{\mathrm{orb}}(d)$ with $\mathcal E_{d}=0$ in the
tabulated range, Proposition~\ref{prop:odd3}):
\begin{center}\resizebox{\textwidth}{!}{$
\begin{array}{c|cccccccccccccc}
d & 4&6&8&10&12&14&16&18&20&22&24&26&28&30\\\hline
1+\#\mathcal O_d & 5&14&14&26&38&38&35&74&62&62&92&74&86&170\\
\mathcal E_{d} & 0&0&0&0&0&8&0&0&0&0&6&0&8&6
\end{array}$}\end{center}
\noindent(Provenance of each layer: odd $d$ --- Burnside and the closed form,
Theorem~\ref{thm:closedform}/Proposition~\ref{prop:odd3}; the even-$d$ decomposable count
$\#\mathcal O_d$ --- the stabilizer descent, Proposition~\ref{prop:cocycle}; $\mathcal E_{d}$ --- the
internal identities of Theorem~\ref{thm:qr} together with the exact per-orbit certificates.)
For higher degrees with $\gcd(d,6)=1$ the closed form of Theorem~\ref{thm:closedform} applies verbatim;
the ancillary script \texttt{asymptotic\_rho\_Q.py} prints the resulting exact values for all such
degrees $5\le d\le100$ (in the ancillary archive; not reprinted here).

\section{\texorpdfstring{$\NS(X_d)$ as a Galois module: $\rho_K$ for $K\subseteq\Q(\zeta_d)$, a reduction for all $K$}{NS(X\_d) as a Galois module: rho\_K for K in Q(zeta\_d), a reduction for all K}}\label{sec:module}
The reduction \eqref{eq:orbitrule} upgrades to a module statement. By Shioda's eigenspace decomposition
each $V(a)$ (the $E$-form of \S\ref{sec:even}) is a line over $\Q(\zeta_d)$; $\Gal(\Q(\zeta_d)/\Q)=(\Z/d)^\times$ permutes the
$V(a)$ by $a\mapsto ta$, while $\Gal(\Qbar/\Q(\zeta_d))$ acts on each $V(a)$ through the
Weil/Gr\"ossencharacter twist $u_a$ --- the same Gr\"ossencharacter \cite{GCS} use to pin down $L$. Hence
the $\zeta_d$-rational N\'eron--Severi group is
\[
\NS\bigl(X_{\Q(\zeta_d)}\bigr)\otimes\Q \;=\; \Q h\;\oplus\; M,\qquad
M\otimes_\Q E\;=\;\bigoplus_{a\in\mathcal S}V(a)\quad(E=\Q(\zeta_d)),
\]
with $h$ the hyperplane class and $M$ a $\Q$-form (as in \S\ref{sec:even}) that becomes, after $\otimes\,E$, the twisted permutation
(monomial) module on the eigenlines; its character is $\chi_M(t)=\tr(\sigma_t\mid M)$, and we write
$\chi_{\NS}(t)=1+\chi_M(t)$ for the character of the full group. For a $\Q$-representation,
$\dim(\cdot)^H=\tfrac1{|H|}\sum_{t\in H}\chi(t)$ is the standard
dimension of the $H$-invariants.

\begin{theorem}[Cyclotomic tower]\label{thm:J1}
For $\Q\subseteq K\subseteq\Q(\zeta_d)$ with $H=\Gal(\Q(\zeta_d)/K)$,
\[
\rho_K(X_d)\;=\;\frac1{|H|}\sum_{t\in H}\chi_{\NS}(t).
\]
This recovers $\rho_\Q$ ($H=(\Z/d)^\times$) and $\rho_{\Q(\zeta_d)}$ ($H=\{1\}$) and every field between.
For $\gcd(d,6)=1$ the twist is trivial (Theorem~\ref{thm:dichotomy}$(\Leftarrow)$, via Lemma~\ref{lem:pairing};
cf.\ \cite{GCS,Degtyarev}) and $\rho_K=1+\#\{H\text{-orbits of }\A\}$.
\end{theorem}

\begin{theorem}[All number fields]\label{thm:J2}
$\NS(X_{\Qbar})\otimes\Q$ is defined over the field of definition $L$ of \cite{GCS}; $L/\Q$ is finite
Galois, and the $\Gal(\Qbar/\Q)$-action factors through $\Gal(L/\Q)$ by the definition of $L$. (Any
finite Galois $L'/\Q$ with $\Gal(\Qbar/L')$ acting trivially on $\NS(X_{\Qbar})\otimes\Q$ would serve;
$L$, which trivializes the integral Picard group, is such a field but need not be minimal for the
rational representation.) Hence
for every number field $K$ (the intersection $K\cap L$ taken inside the fixed $\Qbar$),
\[
\rho_K(X_d)\;=\;\rho_{K\cap L}\;=\;\frac1{|\Gal(L/(K\cap L))|}\sum_{\sigma}\chi_{\NS}(\sigma)
\]
(with $\chi_{\NS}$ now the character of the $\Gal(L/\Q)$-action),
reducing to Theorem~\ref{thm:J1} when $K\subseteq\Q(\zeta_d)$. Over $\Q(\zeta_d)$ the permutation is split,
so for $K\supseteq\Q(\zeta_d)$ each $V(a)$ is $K$-rational iff $K$ contains the Kummer generator of its
twist field --- for a family class with parameter $x$, the $\gcd$-reduced radical of its $\beta$
(Lemma~\ref{lem:famnontriv}; e.g.\ $2^{1/3}$ for family (a) at $d=12$, $x=2$, while the \cite{GCS}
radicals $2^{2/d},3^{3/d}$ arise at $x=1$ and generate the family layer jointly), and for the exceptional classes the generators over
$\Q(\zeta_d)$ of the per-character fields $L_\chi$ of \cite[Prop.~6.4, Rem.~6.6]{GCS}; the tower
$\Q(\zeta_d)\subseteq\Q(\zeta_{2d})\subseteq L$ (the middle being the $\mu_{2d}$ step, nontrivial for even
$d$, which trivializes the parity twist of Theorem~\ref{thm:parity}) reaches $\rho_{\mathrm{geom}}$ at $K=L$.
The reduction $\rho_K=\rho_{K\cap L}$ holds for every number field $K$ and every $d$. What is explicit
is delimited as follows: for $4\le d\le30$ the full character $\chi_{\NS}$ is tabulated (ancillary
\texttt{character\_table.csv}), so Theorem~\ref{thm:J1} evaluates $\rho_K$ for every
$K\subseteq\Q(\zeta_d)$ in that range, the exceptional contributions at $d=14,24,28,30$ being decided
by Theorem~\ref{thm:qr} and the exact certificates; the radical towers above $\Q(\zeta_d)$ are
evaluated to $\rho_{\mathrm{geom}}$ at the displayed degrees $d=8,9,16$; and for $\gcd(d,6)=1$ the
closed form of Theorem~\ref{thm:closedform} gives $\rho_\Q$ for every degree.
\end{theorem}

The method produces, e.g.,
\begin{center}\resizebox{\textwidth}{!}{$
\rho_{\Q(i)}(X_8)=26,\quad \rho_{\Q(\sqrt2)}(X_8)=\rho_{\Q(\sqrt{-2})}(X_8)=29,\quad \rho_{\Q(\sqrt5)}(X_{10})=50,
$}\end{center}
the six cyclotomic-subfield ranks of $X_{25}$ --- indexed by the subfields
$K\subseteq\Q(\zeta_{25})$ of degree $1,2,4,5,10,20$ over $\Q$ --- being
$91$, $181$, $361$, $415$, $829$, $1657$, and the radical tower of
Theorem~\ref{thm:J2}:
\begin{gather*}
\rho_{\Q(\zeta_8)}=56\to\rho_{\Q(\zeta_{16})}=128\to\rho_L=176\ (d=8);\qquad
\rho_{\Q(\zeta_9)}=169\to\rho_L=217\ (d=9);\\
\rho_{\Q(\zeta_{16})}=296\to\rho_{\Q(\zeta_{32})}=632\to\rho_L=872\ (d=16)
\end{gather*}
(at $d=9$, $\Q(\zeta_{18})=\Q(\zeta_9)$). The character
$(56,2,-4,2)$ underlying the $d=8$ case is produced by the module computation above; the descent
mechanism it instantiates is checked independently, and disjointly from Gauss sums, on the quartic $X_4$
by its line geometry
(ancillary \texttt{lines\_descent\_quartic\_check.py}), where the $48$ lines generate $\NS$.

\section{\texorpdfstring{The average order of $\rho_\Q$}{The average order of rho\_Q}}\label{sec:asymp}

\begin{theorem}\label{thm:asymp}
$\displaystyle\sum_{d\le x,\ \gcd(d,6)=1}\rho_\Q(X_d)\ \sim\ \tfrac35\,x^2$; equivalently, the mean value
$\bigl(\#\{d\le x:\gcd(d,6)=1\}\bigr)^{-1}\!\!\sum_{d\le x,\,\gcd(d,6)=1}\rho_\Q(X_d)\sim\tfrac95\,x$
(the average order in the summatory sense, $\sum_{d\le x}$-matched, being $\tfrac{18}5\,d$).
\end{theorem}
\begin{proof}
By Theorem~\ref{thm:closedform}, $\rho_\Q(X_d)=3\Psi_2(d)+O(\tau(d))$ and $\sum_{d\le x}\tau(d)=O(x\log x)
=o(x^2)$. The Dirichlet series of the multiplicative $\Psi_2$ is
$\sum_n\Psi_2(n)n^{-s}=\zeta(s-1)\zeta(s)^2/\zeta(2s)$, with a simple pole at $s=2$ of residue
$\zeta(2)^2/\zeta(4)=\tfrac52$; restricting to $\gcd(n,6)=1$ divides the residue by the $p=2,3$ Euler
factors (product $\tfrac{25}4$ at $s=2$), giving $\tfrac25$, so $\sum_{n\le x,(n,6)=1}\Psi_2(n)\sim\tfrac15
x^2$ by a standard Tauberian theorem \cite{Delange}. (The Tauberian input can also be avoided
elementarily: $\Psi_2=\mathrm{Id}\ast h$ with $\sum_n h(n)n^{-s}=\zeta(s)^2/\zeta(2s)$ absolutely
convergent at $s=2$, so partial summation gives the same asymptotic.) The constant is confirmed numerically to four digits.
\end{proof}

\noindent The Galois action thus removes almost all of $\rho_{\mathrm{geom}}$: the defect
$\rho_{\mathrm{geom}}-\rho_\Q$ equals $3(p-2)^2$ for prime $p\ge5$, of order $\rho_{\mathrm{geom}}\sim3p^2$,
while $\rho_\Q\sim3p$.

\section{Provenance, verification, and scope}\label{sec:prov}
\emph{Status summary.} Every result is unconditional. The inputs divide into three classes:
(i) the classification of algebraic characters (\cite{Aoki83}, list \cite{Terasoma}; re-verified by
independent enumeration as described below); (ii) proofs in the text (everything else); and (iii) one
computer-assisted component --- the completeness of the exceptional-orbit \emph{elimination} for
$4\le d\le30$, with an exact $\Z[\zeta_m]$ witness per removed orbit, every surviving contribution
being proved by Theorem~\ref{thm:qr}.

The surface classification (decomposable $\cup$ three families $\cup$ $101$ exceptional Galois orbits across
$22$ levels $\le180$, odd levels only $15,21$) is Shioda's conjecture \cite{Shioda81} proved by Aoki
\cite{Aoki83}; the three families, in the representative form of Theorem~\ref{thm:families}, are
(a) $\{2x,-x,-x+\tfrac d2,\tfrac d2\}$ for $2\mid d$, (b) $\{x,x+\tfrac d3,x+\tfrac{2d}3,-3x\}$ for
$3\mid d$, and (c) $\{4x,-2x,-x+\tfrac d4,-x+\tfrac{3d}4\}$ for $4\mid d$, with $x$ ranging over the
residues making all entries nonzero (an equivalent parameterization: \cite[proof of Prop.~6.3]{GCS});
we use Terasoma's statement and list \cite{Terasoma} and have re-verified it by independent
enumeration and exhaustively at $d=9,15,21$. (The same $22$ levels index the field-of-definition fields
$M_{d'}$ of \cite{GCS}.)

\emph{Independent verification.} The Weil--Jacobi eigenvalue formula underlying the table was
cross-checked against brute-force $\F_p$ point-counting --- a computation disjoint from the Gauss-sum
machinery --- for all
$d=4$--$30$ (three split primes each): the total counts $\#X_d(\F_p)$ match the formula's aggregate
exactly (an aggregate check, not a per-eigenvalue certificate); the per-character Jacobi values then
reproduce the geometric
Picard numbers $1+\#\A$ ($20,37,86,\dots$) and the exact $\zeta_d$-ranks ($8,37,26,91,56$ for
$d=4$--$8$). (Split
primes are trivial in $\Gal(\Q(\zeta_d)/\Q)$, so this corroborates the aggregate formula at the
tested primes ($\rho_{\Q(\zeta_d)}$ itself comes from the per-character exact evaluation), and in no
case the descent to $\Q$ --- for which see the $X_4$ line check below.) All
twists are computed as exact objects (Gauss-sum products, Hasse--Davenport reductions, Kummer theory);
an \emph{exact} evaluation at one split prime proves nontriviality; floating-point evaluations are used
only as corroboration of exact identities (below $10^{-30}$ in the high-precision twist checks; the
point-count layer runs in double precision and is decided by exact integer equality, its rounding
error, $\sim10^{-8}$, lying far below the $1/2$ margin). The even-degree descent to $\Q$ (Proposition~\ref{prop:cocycle}) is corroborated by a second,
geometry-only route: for the quartic $X_4$, whose $48$ lines generate $\NS$, the tower
$\rho_\Q=5<\rho_{\Q(i)}=8<\rho_{\mathrm{geom}}=20$ is recovered from the line-intersection form
and the Galois permutation of lines alone. The exceptional contributions $\mathcal E_{24}=6,\mathcal E_{28}=8,\mathcal E_{30}=6$ are
proved internally (Theorem~\ref{thm:qr}: telescope, sub-level reduction, and a direct Hasse--Davenport
identity) and corroborated three ways: by exact $\Z[\zeta_m]$ evaluation of the Jacobi-sum twists (which
rigorously proves \emph{non}triviality at split primes), by an independent uniform Gauss-sum
recomputation reproducing
$\rho_\Q(X_{24})=98,\ \rho_\Q(X_{28})=94,\ \rho_\Q(X_{30})=176$, and by the per-character
field-of-definition computation of \cite{GCS} (their kernels $L_\chi$, \cite[Prop.~6.4, Rem.~6.6]{GCS};
the published output records the aggregate $M_{d'}$, consistent with the counts).

\emph{Algebraicity.} All results here are \emph{unconditional}: on the surface $X_d$ the algebraic classes
are exactly the rational $(1,1)$-classes by the Lefschetz theorem, so the eigenlines indexed by $\A$
exhaust the primitive N\'eron--Severi group after $\otimes\,\mathbb{C}$.

\medskip
{\sloppy
\emph{Code availability.} Standalone Python scripts reproducing every table and each load-bearing
computation are provided as ancillary files with this submission (see \texttt{README.md}): the $\F_p$
point-count verification, the closed form, the even-degree stabilizer-cocycle descent
(Proposition~\ref{prop:cocycle}, \texttt{evendescent\_cocycle\_check.py}), the exact $\Z[\zeta_m]$
exceptional evaluation (\texttt{exceptional\_exact\_twist.py}, rigorous for non-triviality; the
surviving-class trivialities are proved in Theorem~\ref{thm:qr}), the geometry-only quartic tower
(\texttt{lines\_descent\_quartic\_check.py}), the field-of-definition grounding against \cite{GCS}
(\texttt{gcs\_fieldofdef\_grounding.py}), the Galois-module ranks $\rho_K$, the radical tower, and the
average order. Two machine-readable certificates accompany the scripts
(\texttt{make\_certificates.py}): \texttt{exceptional\_certificates.csv} --- one row per
indecomposable orbit at all nine exceptional-adjacent levels $12,14,15,18,20,21,24,28,30$
(representative, orbit length, tuple
stabilizer, disposition with its exact witness or identity, contribution to $\mathcal E_{m}$ (the CSV column writes it as \texttt{E\_m}); $1270$ rows,
with $\mathcal E_{12}=\mathcal E_{18}=\mathcal E_{20}=0$ certified row-by-row; the completeness of the
exceptional-orbit elimination for $4\le d\le30$ is thus \emph{computer-assisted} --- each removed orbit
carries an exact $\Z[\zeta_m]$ Frobenius witness, while every surviving contribution is proved by the
displayed identities of Theorem~\ref{thm:qr}) ---
and \texttt{character\_table.csv}, the full character vectors
$(\chi_{\NS}(t))_{t\in(\Z/d)^\times}$ for $4\le d\le30$, whose per-degree averages reproduce the
entire $\rho_\Q$ table and whose values at $t=1$ reproduce every $\rho_{\Q(\zeta_d)}$.
Dependency versions are pinned (\texttt{requirements.txt}) and every ancillary code and data file
is checksummed (\texttt{SHA256SUMS}). Each script runs on a stock scientific Python stack (\texttt{numpy}, \texttt{sympy},
\texttt{mpmath}) and prints its own checks: exact identities are verified symbolically or by exact integer
arithmetic, and the remaining numerical corroborations have absolute discrepancy below $10^{-30}$
(high-precision checks) or are decided by exact integer equality with rounding error far below the
$1/2$ margin (the double-precision point-count layer).\par}

\emph{On the use of AI.} This work was carried out with substantial AI assistance in derivation,
computation, and drafting. Every load-bearing computational claim is accompanied by a reproducible
script or certificate, and the general mathematical statements are proved in the text or explicitly
reduced to cited theorems (the exceptional trivialities in Theorem~\ref{thm:qr}, with \cite{GCS} and
the scripts corroborating);
the load-bearing results were independently reproduced --- including the
descent to $\Q$, by a disjoint line-geometry computation on the quartic $X_4$ (its $48$ lines and their
intersection form), and the exceptional values by an independent uniform Gauss-sum recomputation. The
author takes full responsibility for the contents.

\end{document}